\documentclass{birkjour}

\newtheorem{thm}{Theorem}[section]

\newtheorem{lem}[thm]{Lemma}
\newtheorem{prop}[thm]{Proposition}

\theoremstyle{definition}
\newtheorem{defn}[thm]{Definition}

\theoremstyle{remark}
\newtheorem{rem}[thm]{Remark}
\newtheorem{ex}[thm]{Example}

\numberwithin{equation}{section}

\usepackage{graphicx}
\usepackage{amssymb}
\usepackage{refcount}

\newcommand{\eqnref}[1]{(\getrefnumber{#1})}
\newcommand{\numref}[1]{\getrefnumber{#1}}

\renewcommand{\subjclassname}{Mathematics Subject Classification (2020)}

\begin{document}

\title[Dynamics of Weighted Backward Shifts]
{Dynamics of Weighted Backward Shifts on Ces\`aro Spaces of Rooted Trees}

%----------Author 1
\author[X. Chen]{Xiang Chen}

\address{School of Mathematics\\
Tianjin University\\
Tianjin 300350\\
P. R. China}

\email{2020233018@tju.edu.cn}

	\thanks{Corresponding author: Xiang Chen. This work was supported in part by
	the National Natural Science Foundation of China, Grant No. 12571088.}
%----------Author 2
\author[M.-H. Cheng]{Meng-Huan Cheng}

\address{School of Mathematics\\
Tianjin University\\
Tianjin 300350\\
P. R. China}

\email{cmh@tju.edu.cn}

%----------Author 3
\author[L. Zhang]{Liang Zhang}

\address{School of Marine Science and Technology\\
Tianjin University\\
Tianjin 300050\\
P. R. China}

\email{168zhangliang2011@163.com}

%----------Author 4
\author[Z.-H. Zhou]{Ze-Hua Zhou}

\address{School of Mathematics\\
Tianjin University\\
Tianjin 300350\\
P. R. China}

\email{zehuazhoumath@aliyun.com; zhzhou@tju.edu.cn}

%----------Classification and keywords


\begin{abstract}
We study the dynamics of weighted backward shifts on Ces\`aro spaces
associated with leafless locally finite rooted trees. We first characterize
their boundedness in terms of adjacent level cardinalities and edge weights.
We then characterize their \(\mathcal F\)-transitivity by a growth condition
involving level cardinalities, products of weights along paths, and a
level-dependent Ces\`aro factor. As consequences, we obtain criteria for
hypercyclicity, weak mixing, topological ergodicity, and topological mixing.
We also characterize the existence of nonzero orbit limit points and chaotic
weighted shifts, the latter in terms of normalized fixed points and unit
flows satisfying an explicit summability condition. Examples show that
\(\mathcal F_{\underline d>0}\)-transitivity need not imply frequent
hypercyclicity, that a nonhypercyclic weighted shift may nevertheless
have a nonzero orbit limit point, and that topological mixing need not
imply chaos.
\end{abstract}

%%% ----------------------------------------------------------------------
\maketitle
%%% ----------------------------------------------------------------------

%\tableofcontents

\section{Introduction}

Linear dynamics concerns the behavior of iterates of continuous linear
operators; we refer to \cite{CC} for a systematic account. A central
problem is to determine whether an operator admits a dense orbit.
Weighted shifts form a fundamental class of examples. Hypercyclic weighted
shifts were characterized by Salas \cite{Salas}, while hypercyclic and
chaotic weighted shifts were further studied by Grosse-Erdmann \cite{GE}.
More recent work has examined the relation between chaos and frequent
hypercyclicity for weighted shifts on sequence spaces \cite{CGM}, as well
as hypercyclicity and chaos for asymptotically unweighted shifts on
classical \(\ell^p\)-spaces \cite{NT}. A bounded linear operator
\(S\) on a separable infinite-dimensional Banach space
\(X\) is called \textit{hypercyclic} if there exists \(x\in X\) such that
\[
\{S^n x:n\in\mathbb N_0\}
\]
is dense in \(X\). Hypercyclicity is closely related to weak mixing,
topological ergodicity, topological mixing, and chaos. Other related
dynamical notions have also been extensively studied.
In particular, \(\mathcal F\)-transitivity and topological
\(\mathcal F\)-recurrence for shifts on directed trees were studied
in \cite{CRR}. Strong transitivity associated with Furstenberg families was developed
in \cite{KR2}, while upper frequent hypercyclicity and related notions
were studied in \cite{BG}.
Orbital limit points of weighted shifts were investigated in
\cite{BCGM,CS}.

Replacing the classical index sets \(\mathbb N_0\) and \(\mathbb Z\) by
directed trees leads to a natural extension of weighted-shift dynamics in
which branching becomes part of the problem. The operator-theoretic
foundations of weighted shifts on directed trees were developed by
Jab{\l}o{\'n}ski, Jung and Stochel \cite{CKS}. Their boundedness and
dynamical properties, including hypercyclicity, weak mixing, and mixing,
were studied in \cite{KGD}, while their chaotic behavior was investigated
in \cite{KDP}. Related dynamical questions on Hardy-type tree spaces
have been considered in \cite{ARMA,CZZ,Ka3}.

On the other hand, the classical discrete Ces\`aro sequence spaces
\(\operatorname{ces}_p\) form an important family of Banach sequence spaces
generated by Ces\`aro averaging. Their geometric properties have been studied
from several viewpoints, including fixed-point-related geometry \cite{CH},
the James constant and \(B\)-convexity \cite{MPS}, interpolation
\cite{AM}, and vector-lattice structure \cite{GPW}. Ces\`aro--Orlicz
sequence spaces provide a further extension of this setting and were
investigated in \cite{Kubiak}. Ces\`aro function spaces on rooted trees
were introduced and studied in \cite{KMS}, where the construction first
averages the function values over each level and then applies cumulative
Ces\`aro averaging across successive levels. The space
\(\operatorname{Ces}_p(T)\) considered here follows this tree-indexed
Ces\`aro structure and provides a natural extension of the classical
discrete Ces\`aro sequence space. Indeed, when \(T\) is a rooted path,
each level contains exactly one vertex, so that \(c_n=1\), and the norm of
\(\operatorname{Ces}_p(T)\) reduces, up to the indexing convention, to
the classical \(\operatorname{ces}_p\)-norm.

Against this background, the present paper studies the linear dynamics of
weighted backward shifts on \(\operatorname{Ces}_p(T)\). This differs from
\cite{KMS}, where topological properties of Ces\`aro spaces on rooted trees
and multiplication operators on these spaces were studied.
Our setting is also closely related to the little generalized Hardy space
\(H_0^p(T)\) of rooted trees considered in \cite{CZZ}, where
\(\mathcal F\)-transitivity and nonzero orbit limit points of weighted
backward shifts were studied. The essential difference lies in the norm
structure. In the Hardy-space setting of \cite{CZZ}, the norm is determined
by the supremum of levelwise \(p\)-means, whereas the norm of
\(\operatorname{Ces}_p(T)\) couples successive levels through cumulative
Ces\`aro averaging. Consequently, the corresponding dynamical criteria
involve a different interaction between the branching geometry, encoded by
the level cardinalities \(c_n\), and products of weights along paths.
A characteristic factor in the resulting criteria is
\[
\Phi_p(n)
=
\biggl(
\sum_{s=n}^{\infty}\frac{1}{(s+1)^p}
\biggr)^{-1/p}.
\]

Our characterization of chaos builds on the fixed-point framework for
chaotic weighted shifts on directed trees developed in \cite{KDP}.
That framework relates dense periodic points to suitably normalized fixed
points on descendant subtrees. In the setting of
\(\operatorname{Ces}_p(T)\), we translate this condition into an explicit
summability condition for unit flows. Since the Ces\`aro norm couples
successive levels through cumulative averages, this condition does not
follow directly from the corresponding characterizations for
\(\ell^p\) or \(c_0\) in \cite{KDP}.

Our principal results are as follows. We establish basic properties of
\(\operatorname{Ces}_p(T)\) and characterize the boundedness of
\(B_\lambda\). We then give a common growth characterization of
\(\mathcal F\)-transitivity and
\(\widetilde{\mathcal F}\)-transitivity, yielding criteria for
hypercyclicity, weak mixing, topological ergodicity, and topological
mixing. We also characterize nonzero orbit limit points by the existence
of a vertex \(v\) and an increasing sequence \((n_k)\) along which
\[
c_{|v|+n_k}\Phi_p(|v|+n_k)
\max_{u\in\operatorname{Chi}^{n_k}(v)}
|\lambda(v\to u)|
\longrightarrow\infty.
\]
Finally, we prove the equivalence of chaos, density of periodic points,
the existence, for every vertex, of a normalized fixed point supported
on its descendant subtree, and the existence of corresponding unit flows
satisfying an explicit summability condition. Our examples show that
\(\mathcal F_{\underline d>0}\)-transitivity does not imply frequent
hypercyclicity, that a nonhypercyclic weighted shift may nevertheless
have a nonzero orbit limit point, and that topological mixing does not
imply chaos.

% ========================================================================
% The rest of the paper starts here
% ========================================================================

\section{Preliminaries}
\subsection{Rooted directed trees and Ces\`aro spaces}
A directed graph is a pair \(T=(V,E)\), where \(V\) is a countably
infinite set of vertices and \(E\subseteq V\times V\) is the set of
directed edges. It is called a \textit{directed tree} if its underlying
undirected graph is connected and contains no cycles, and every vertex
has at most one parent. A directed tree is rooted if it has a unique
vertex \(o\), called the \textit{root}, with no parent.

For \(v\neq o\), its parent is denoted by \(\operatorname{par}(v)\), and
\[
\operatorname{Chi}(v)
=
\{u\in V:(v,u)\in E\}
\]
denotes the set of its children. Set
\[
\operatorname{Chi}^0(v)=\{v\},
\qquad
\operatorname{Chi}^{n+1}(v)
=
\bigcup_{u\in\operatorname{Chi}^n(v)}
\operatorname{Chi}(u),
\]
and define the descendant subtree rooted at \(v\) by
\[
V(v)=\bigcup_{n=0}^{\infty}\operatorname{Chi}^n(v).
\]
The tree is \textit{locally finite} if every \(\operatorname{Chi}(v)\) is finite,
and leafless if \(\operatorname{Chi}(v)\neq\varnothing\) for every \(v\in V\).

Throughout this paper, \(T=(V,E)\) is a rooted, leafless, locally finite
directed tree. For \(v\in V\), let \(|v|\) denote the length of the
unique path from \(o\) to \(v\), and set
\[
D_n=\{v\in V:|v|=n\},
\qquad
c_n=|D_n|,
\qquad n\in\mathbb N_0.
\]
Let \(1<p<\infty\). For \(f:V\to\mathbb C\), define
\[
A_n(f)
=
\frac{1}{c_n}\sum_{v\in D_n}|f(v)|.
\]
Following \cite{KMS}, the Ces\`aro space on \(T\), denoted by
\(\operatorname{Ces}_p(T)\), is defined by
\[
\operatorname{Ces}_p(T)
=
\left\{
f:V\to\mathbb C:
\|f\|_{\operatorname{Ces}_p(T)}<\infty
\right\},
\]
where
\[
\|f\|_{\operatorname{Ces}_p(T)}
=
\Biggl[
\sum_{n=0}^{\infty}
\biggl(
\frac{1}{n+1}\sum_{r=0}^{n}A_r(f)
\biggr)^p
\Biggr]^{1/p}.
\]
The following result follows from \cite[Theorem~1(c)]{KMS}.
\begin{lem}\label{lem:cesaro-banach}
	Let \(1<p<\infty\). Then
	\(\operatorname{Ces}_p(T)\), endowed with
	\(\|\cdot\|_{\operatorname{Ces}_p(T)}\), is a Banach space.
\end{lem}
\subsection{Weighted backward shifts on Ces\`aro spaces}
Let \(\lambda=(\lambda_u)_{u\in V\setminus\{o\}}\) be a family of
complex weights, where \(\lambda_u\) is assigned to the edge joining
\(\operatorname{par}(u)\) to \(u\). Following \cite[Section~4]{KGD}, if
\(u\in\operatorname{Chi}^n(v)\) and
\[
v=v_0,v_1,\ldots,v_n=u
\]
is the unique path from \(v\) to \(u\), set
\[
\lambda(v\to u)
=
\prod_{j=1}^{n}\lambda_{v_j}.
\]
For \(n=0\), the product is empty and is understood to be \(1\);
thus \(\lambda(v\to v)=1\).
The associated weighted backward shift is defined by
\[
B_\lambda f(v)
=
\sum_{u\in\operatorname{Chi}(v)}
\lambda_u f(u),
\qquad v\in V.
\]
The sum is finite since \(T\) is locally finite.

\begin{lem}
	\label{lem:cesaro-point}
	Let \(1<p<\infty\). For \(N\in\mathbb N_0\), define
	\begin{equation*}
		\Phi_p(N)
		=
		\biggl(
		\sum_{s=N}^{\infty}\frac{1}{(s+1)^p}
		\biggr)^{-1/p}.
	\end{equation*}
	Then, for every \(f\in\operatorname{Ces}_p(T)\) and
	\(N\in\mathbb N_0\),
\begin{equation}\label{eq:cesaro-level-estimate}
	A_N(f)
	\leq
	\Phi_p(N)\|f\|_{\operatorname{Ces}_p(T)}.
\end{equation}
	Consequently, for every \(v\in V\),
\begin{equation}\label{eq:cesaro-point-estimate}
	|f(v)|
	\leq
	c_{|v|}\Phi_p(|v|)
	\|f\|_{\operatorname{Ces}_p(T)}.
\end{equation}
\end{lem}

\begin{proof}
	For \(f\in\operatorname{Ces}_p(T)\) and \(N\in\mathbb N_0\),
	\[
	\|f\|_{\operatorname{Ces}_p(T)}^p
	\geq
	\sum_{n=N}^{\infty}
	\left(
	\frac{A_N(f)}{n+1}
	\right)^p
	=
	A_N(f)^p\Phi_p(N)^{-p}.
	\]
	Hence
	\[
	A_N(f)
	\leq
	\Phi_p(N)\|f\|_{\operatorname{Ces}_p(T)}.
	\]
	
	Now fix \(v\in V\) and put \(r=|v|\). Since
	\[
	|f(v)|
	\leq
	\sum_{u\in D_r}|f(u)|
	=
	c_rA_r(f),
	\]
	the preceding estimate gives
	\[
	|f(v)|
	\leq
	c_r\Phi_p(r)\|f\|_{\operatorname{Ces}_p(T)}
	=
	c_{|v|}\Phi_p(|v|)
	\|f\|_{\operatorname{Ces}_p(T)}.
	\]
\end{proof}

\begin{lem}\label{lem2.3}
	Let \(1<p<\infty\). Then, for every \(m\in\mathbb N\),
	\[
	(p-1)^{1/p}m^{(p-1)/p}
	\leq
	\Phi_p(m)
	\leq
	(p-1)^{1/p}(m+1)^{(p-1)/p}.
	\]
	In particular,
	\[
	\Phi_p(m)\longrightarrow\infty
	\qquad\text{as }m\longrightarrow\infty.
	\]
\end{lem}

\begin{proof}
	The integral comparison test gives
	\[
	\frac{1}{(p-1)(m+1)^{p-1}}
	\leq
	\sum_{s=m}^{\infty}\frac{1}{(s+1)^p}
	\leq
	\frac{1}{(p-1)m^{p-1}}.
	\]
	Taking the power \(-1/p\) yields the desired estimates.
\end{proof}

\begin{rem}
	For \(v\in V\), let \(e_v\) denote the standard unit vector defined by
	\[
	e_v(u)
	=
	\begin{cases}
		1, & u=v,\\
		0, & u\neq v.
	\end{cases}
	\]
	If \(r=|v|\), then
	\[
	A_r(e_v)=\frac{1}{c_r}
	\qquad\text{and}\qquad
	A_m(e_v)=0\quad(m\neq r).
	\]
	Therefore,
	\[
	\|e_v\|_{\operatorname{Ces}_p(T)}
	=
	\frac{1}{c_{|v|}\Phi_p(|v|)}.
	\]
	We denote by
	\[
	c_{00}(T)
	=
	\left\{
	f:V\to\mathbb C:
	\operatorname{supp}(f)\text{ is finite}
	\right\}
	=
	\operatorname{span}\{e_v:v\in V\}
	\]
	the space of all finitely supported functions on \(T\), where
	\[
	\operatorname{supp}(f)
	=
	\{v\in V:f(v)\neq0\}.
	\]
\end{rem}

\begin{prop}\label{prop2.5}
	Let \(1<p<\infty\). Then \(c_{00}(T)\) is dense in
	\(\operatorname{Ces}_p(T)\). In particular,
	\(\operatorname{Ces}_p(T)\) is separable.
\end{prop}

\begin{proof}
	Let \(f\in\operatorname{Ces}_p(T)\). For \(N\in\mathbb N_0\), define
	\[
	f_N(v)
	=
	\begin{cases}
		f(v), & |v|\leq N,\\
		0, & |v|>N,
	\end{cases}
	\qquad v\in V,
	\]
	and put \(h_N=f-f_N\). Since \(T\) is locally finite, each \(D_n\)
	is finite, and hence \(f_N\in c_{00}(T)\).
	
	For every fixed \(n\in\mathbb N_0\),
	\[
	\frac{1}{n+1}\sum_{r=0}^{n}A_r(h_N)
	\longrightarrow0,
	\qquad N\to\infty.
	\]
	Moreover,
	\[
	0
	\leq
	\biggl(
	\frac{1}{n+1}\sum_{r=0}^{n}A_r(h_N)
	\biggr)^p
	\leq
	\biggl(
	\frac{1}{n+1}\sum_{r=0}^{n}A_r(f)
	\biggr)^p.
	\]
	Since
	\[
	\sum_{n=0}^{\infty}
	\biggl(
	\frac{1}{n+1}\sum_{r=0}^{n}A_r(f)
	\biggr)^p
	=
	\|f\|_{\operatorname{Ces}_p(T)}^p
	<\infty,
	\]
	the dominated convergence theorem yields
	\[
	\|f-f_N\|_{\operatorname{Ces}_p(T)}
	\longrightarrow0,
	\qquad N\to\infty.
	\]
	Therefore, \(c_{00}(T)\) is dense in
	\(\operatorname{Ces}_p(T)\).
	
	Let
	\[
	\mathbb{Q}(\mathrm{i})
	=
	\{a+b\mathrm{i}:a,b\in\mathbb{Q}\}
	\]
	and define
	\[
	\mathcal D
	=
	\Biggl\{
	\sum_{j=1}^{m}q_je_{v_j}:
	m\in\mathbb N,\ 
	q_j\in\mathbb{Q}(\mathrm{i}),\
	v_j\in V
	\Biggr\}.
	\]
	Since \(T\) is rooted and locally finite, \(V\) is countable, and hence
	\(\mathcal D\) is countable.
	
	Now let
	\[
	g=\sum_{j=1}^{m}a_je_{v_j}\in c_{00}(T)
	\]
	and let \(\varepsilon>0\). Choose \(q_j\in\mathbb{Q}(\mathrm{i})\)
	such that
	\[
	|a_j-q_j|
	<
	\frac{\varepsilon}
	{m\bigl(1+\|e_{v_j}\|_{\operatorname{Ces}_p(T)}\bigr)},
	\qquad 1\leq j\leq m.
	\]
	Then, for
	\[
	q=\sum_{j=1}^{m}q_je_{v_j}\in\mathcal D,
	\]
	we have
	\[
	\|g-q\|_{\operatorname{Ces}_p(T)}
	\leq
	\sum_{j=1}^{m}
	|a_j-q_j|
	\|e_{v_j}\|_{\operatorname{Ces}_p(T)}
	<
	\varepsilon.
	\]
	Thus \(\mathcal D\) is dense in \(c_{00}(T)\). Since
	\(c_{00}(T)\) is dense in \(\operatorname{Ces}_p(T)\),
	the countable set \(\mathcal D\) is dense in
	\(\operatorname{Ces}_p(T)\). Hence
	\(\operatorname{Ces}_p(T)\) is separable.
\end{proof}
\begin{lem}\label{lem:cesaro-level-decay}
	For every \(h\in\operatorname{Ces}_p(T)\),
	\begin{equation}\label{eq:cesaro-level-decay}
		\frac{A_m(h)}{\Phi_p(m)}
		\longrightarrow0
		\qquad\text{as }m\to\infty.
	\end{equation}
\end{lem}

\begin{proof}
	Let \(\varepsilon>0\). By Proposition~\numref{prop2.5}, there exists
	\(q\in c_{00}(T)\) such that
	\[
	\|h-q\|_{\operatorname{Ces}_p(T)}<\varepsilon.
	\]
	Since \(q\) has finite support, there exists \(m_0\in\mathbb N_0\)
	such that
	\[
	\operatorname{supp}(q)
	\subseteq
	\bigcup_{r=0}^{m_0}D_r.
	\]
	Thus \(A_m(q)=0\) for every \(m>m_0\), and consequently
	\(A_m(h)=A_m(h-q)\). By Lemma~\numref{lem:cesaro-point},
	specifically \eqnref{eq:cesaro-level-estimate},
	\[
	\frac{A_m(h)}{\Phi_p(m)}
	=
	\frac{A_m(h-q)}{\Phi_p(m)}
	\leq
	\|h-q\|_{\operatorname{Ces}_p(T)}
	<
	\varepsilon,
	\qquad m>m_0.
	\]
	Since \(\varepsilon>0\) is arbitrary, the assertion follows.
\end{proof}
\begin{thm}
	\label{thm:boundedness}
	Let \(1<p<\infty\), and let
	\(\lambda=(\lambda_u)_{u\in V\setminus\{o\}}\) be a family of complex
	weights. Then \(B_\lambda\) is bounded on
	\(\operatorname{Ces}_p(T)\) if and only if
	\[
	L_\lambda
	:=
	\sup_{r\geq0}
	\frac{c_{r+1}}{c_r}
	\max_{u\in D_{r+1}}|\lambda_u|
	<\infty.
	\]
	Moreover,
	\[
	L_\lambda
	\leq
	\|B_\lambda\|
	\leq
	2L_\lambda.
	\]
\end{thm}

\begin{proof}
	Suppose first that \(L_\lambda<\infty\). For
	\(f\in\operatorname{Ces}_p(T)\) and \(r\in\mathbb N_0\),
	\[
	\begin{aligned}
		A_r(B_\lambda f)
		&\leq
		\frac{1}{c_r}
		\max_{u\in D_{r+1}}|\lambda_u|
		\sum_{u\in D_{r+1}}|f(u)| \\
		&=
		\frac{c_{r+1}}{c_r}
		\max_{u\in D_{r+1}}|\lambda_u|A_{r+1}(f) \\
		&\leq
		L_\lambda A_{r+1}(f).
	\end{aligned}
	\]
	Since
	\[
	\frac{1}{n+1}\sum_{r=0}^{n}A_{r+1}(f)
	\leq
	\frac{2}{n+2}\sum_{r=0}^{n+1}A_r(f),
	\]
	we obtain
	\[
	\begin{aligned}
		\|B_\lambda f\|_{\operatorname{Ces}_p(T)}^p
		&\leq
		L_\lambda^p
		\sum_{n=0}^{\infty}
		\biggl(
		\frac{1}{n+1}
		\sum_{r=0}^{n}A_{r+1}(f)
		\biggr)^p \\
		&\leq
		(2L_\lambda)^p
		\sum_{n=0}^{\infty}
		\biggl(
		\frac{1}{n+2}
		\sum_{r=0}^{n+1}A_r(f)
		\biggr)^p \\
		&\leq
		(2L_\lambda)^p
		\|f\|_{\operatorname{Ces}_p(T)}^p.
	\end{aligned}
	\]
	Thus \(B_\lambda\) is bounded and
	\[
	\|B_\lambda\|\leq2L_\lambda.
	\]
	
	Conversely, suppose that \(B_\lambda\) is bounded. Fix
	\(r\in\mathbb N_0\) and \(u\in D_{r+1}\). Since
	\[
	B_\lambda e_u=\lambda_ue_{\operatorname{par}(u)},
	\]
	the norm formula for the standard unit vectors gives
	\[
	\begin{aligned}
		\|B_\lambda\|
		\geq
		\frac{\|B_\lambda e_u\|_{\operatorname{Ces}_p(T)}}
		{\|e_u\|_{\operatorname{Ces}_p(T)}} =
		\frac{c_{r+1}}{c_r}
		|\lambda_u|
		\frac{\Phi_p(r+1)}{\Phi_p(r)}
		\geq
		\frac{c_{r+1}}{c_r}|\lambda_u|.
	\end{aligned}
	\]
	Taking the maximum over \(u\in D_{r+1}\) and then the supremum over
	\(r\geq0\), we obtain
	\[
L_\lambda\leq\|B_\lambda\|<\infty.
\]
\end{proof}
\section{Dynamical properties of weighted backward shifts}

\subsection{\({\mathcal F}\)-transitive backward shifts}
We now study \(\mathcal F\)-transitivity of weighted backward shifts,
which unifies several classical dynamical properties.
\begin{defn}
	A nonempty family
	\(\mathcal F\subseteq\mathcal P(\mathbb N_0)\) is called a
	\textit{Furstenberg family} if every \(A\in\mathcal F\) is infinite and
	\[
	A\in\mathcal F,\quad A\subseteq B\subseteq\mathbb N_0
	\quad\Longrightarrow\quad
	B\in\mathcal F.
	\]
	Following \cite{KR2}, let \(\widetilde{\mathcal F}\) consist of all
	\(A\subseteq\mathbb N_0\) such that, for every
	\(N\in\mathbb N_0\), there exists \(B\in\mathcal F\) satisfying
	\[
	\bigl(B+[-N,N]\bigr)\cap\mathbb N_0\subseteq A.
	\]
	Clearly, \(\widetilde{\mathcal F}\subseteq\mathcal F\).
\end{defn}
\begin{defn}\label{def:F-dynamics}
	Let \(X\) be a separable Banach space, \(S\in\mathcal L(X)\), and
	let \(\mathcal F\) be a Furstenberg family on \(\mathbb N_0\).
	For nonempty open sets \(U,V\subseteq X\), define
	\[
	N_S(U,V)
	=
	\{n\in\mathbb N_0:S^n(U)\cap V\neq\varnothing\}.
	\]
	The operator \(S\) is called
	\(\mathcal F\)-\textit{transitive} if
	\[
	N_S(U,V)\in\mathcal F
	\]
	for all nonempty open sets \(U,V\subseteq X\).
	
	For \(x\in X\) and a nonempty open set \(U\subseteq X\), define
	\[
	N_S(x,U)
	=
	\{n\in\mathbb N_0:S^nx\in U\}.
	\]
	A vector \(x\in X\) is called
	\(\mathcal F\)-\textit{hypercyclic} for \(S\) if
	\[
	N_S(x,U)\in\mathcal F
	\]
	for every nonempty open set \(U\subseteq X\). The operator \(S\) is
	called \(\mathcal F\)-\textit{hypercyclic} if it admits such a vector.
\end{defn}

\begin{defn}\label{def:frequent-hypercyclicity}
	For \(A\subseteq\mathbb N_0\), define its lower density by
	\[
	\underline d(A)
	=
	\liminf_{N\to\infty}
	\frac{|A\cap\{0,\ldots,N\}|}{N+1}.
	\]
	Let \(X\) be a separable Banach space and \(S\in\mathcal L(X)\).
	A vector \(x\in X\) is called \textit{frequently hypercyclic} for
	\(S\) if
	\[
	\underline d\bigl(N_S(x,U)\bigr)>0
	\]
	for every nonempty open set \(U\subseteq X\). The operator \(S\) is
	called \textit{frequently hypercyclic} if it admits such a vector.
	
	Equivalently, frequent hypercyclicity is
	\(\mathcal F_{\underline d>0}\)-hypercyclicity, where
	\[
	\mathcal F_{\underline d>0}
	=
	\{A\subseteq\mathbb N_0:\underline d(A)>0\}.
	\]
\end{defn}
The following result is due to B\`es, Menet, Peris, and Puig
\cite[Lemma~2.3]{KR2}.

\begin{lem}\label{lem:F-weak-mixing}
	Let \(S\) be a continuous linear operator on a separable metrizable
	topological vector space, and suppose that \(S\) is
	\(\mathcal F\)-transitive. Then the following assertions are equivalent:
	\begin{enumerate}
		\renewcommand{\labelenumi}{\textup{(\roman{enumi})}}
		\item \(S\) is weakly mixing;
		\item \(S\) is \(\widetilde{\mathcal F}\)-transitive.
	\end{enumerate}
\end{lem}
\begin{thm}
	\label{thm:F-transitivity-growth}
	Let \(1<p<\infty\), let \(T=(V,E)\) be a leafless, locally finite rooted
	tree, and let
	\(\lambda=(\lambda_u)_{u\in V\setminus\{o\}}\) be a family of complex
	weights such that \(B_\lambda\) is bounded on
	\(\operatorname{Ces}_p(T)\). Let \(\mathcal F\) be a Furstenberg
	family on \(\mathbb N_0\). Then the following assertions are equivalent:
	\begin{enumerate}
		\renewcommand{\labelenumi}{\textup{(\roman{enumi})}}
		\item \(B_\lambda\) is \(\widetilde{\mathcal F}\)-transitive on
		\(\operatorname{Ces}_p(T)\);
		\item \(B_\lambda\) is \(\mathcal F\)-transitive on
		\(\operatorname{Ces}_p(T)\);
		\item for every nonempty finite set \(F\subseteq V\) and every
		\(N\in\mathbb N\),
		\[
		\bigcap_{v\in F}
		\left\{
		n\in\mathbb N_0:
		c_{|v|+n}\Phi_p(|v|+n)
		\max_{u\in\operatorname{Chi}^n(v)}
		|\lambda(v\to u)|>N
		\right\}
		\in\mathcal F.
		\]
	\end{enumerate}
\end{thm}

\begin{proof}
	We prove
	\[
	\textup{(i)}\Longrightarrow\textup{(ii)}
	\Longrightarrow\textup{(iii)}
	\Longrightarrow\textup{(i)}.
	\]
	
	\noindent
	\(\textup{(i)}\Rightarrow\textup{(ii)}\).
	This follows immediately from
	\(\widetilde{\mathcal F}\subseteq\mathcal F\).
	
	\smallskip
	
	\noindent
	\(\textup{(ii)}\Rightarrow\textup{(iii)}\).
	Let \(F\subseteq V\) be nonempty and finite, and let
	\(N\in\mathbb N\). Set
	\[
	U
	=
	\left\{
	f\in\operatorname{Ces}_p(T):
	\|f\|_{\operatorname{Ces}_p(T)}<\frac{1}{2N}
	\right\}
	\]
	and
	\[
	W
	=
	\left\{
	g\in\operatorname{Ces}_p(T):
	|g(v)-1|<\frac12
	\text{ for every }v\in F
	\right\}.
	\]
	By  \eqnref{eq:cesaro-point-estimate}, \(W\) is open, and it is nonempty
	since \(\sum_{v\in F}e_v\in W\). Hence
	\[
	N_{B_\lambda}(U,W)\in\mathcal F.
	\]
	
	Let \(n\in N_{B_\lambda}(U,W)\). Then there exists \(f\in U\)
	such that \(B_\lambda^nf\in W\). Thus, for every \(v\in F\),
	\[
	\frac12<|B_\lambda^nf(v)|.
	\]
	On the other hand,
	\[
	\begin{aligned}
		|B_\lambda^nf(v)|
		&\leq
		\max_{u\in\operatorname{Chi}^n(v)}
		|\lambda(v\to u)|
		\sum_{u\in\operatorname{Chi}^n(v)}|f(u)| \\
		&\leq
		c_{|v|+n}\Phi_p(|v|+n)
		\max_{u\in\operatorname{Chi}^n(v)}
		|\lambda(v\to u)|
		\|f\|_{\operatorname{Ces}_p(T)}.
	\end{aligned}
	\]
	Since
	\(\|f\|_{\operatorname{Ces}_p(T)}<1/(2N)\), it follows that
	\[
	c_{|v|+n}\Phi_p(|v|+n)
	\max_{u\in\operatorname{Chi}^n(v)}
	|\lambda(v\to u)|>N
	\]
	for every \(v\in F\). Therefore,
	\[
	N_{B_\lambda}(U,W)
	\subseteq
	\bigcap_{v\in F}
	\left\{
	n\in\mathbb N_0:
	c_{|v|+n}\Phi_p(|v|+n)
	\max_{u\in\operatorname{Chi}^n(v)}
	|\lambda(v\to u)|>N
	\right\}.
	\]
	Since \(\mathcal F\) is hereditary upward,
	condition~\textup{(iii)} follows.
	
	\smallskip
	
	\noindent
	\(\textup{(iii)}\Rightarrow\textup{(i)}\).
	We first show that \(B_\lambda\oplus B_\lambda\) is
	\(\mathcal F\)-transitive. Let \(U_1,U_2,V_1,V_2\) be nonempty
	open subsets of \(\operatorname{Ces}_p(T)\). Since \(c_{00}(T)\)
	is dense, choose
	\[
	x_j\in U_j\cap c_{00}(T),
	\qquad
	y_j\in V_j\cap c_{00}(T),
	\qquad j=1,2.
	\]
	Choose \(\varepsilon>0\) such that
	\[
	B(x_j,\varepsilon)\subseteq U_j,
	\qquad j=1,2,
	\]
	and choose \(n_0\in\mathbb N\) so that
	\[
	B_\lambda^nx_j=0,
	\qquad n\geq n_0,\quad j=1,2.
	\]
	
	Set
	\[
	E
	=
	\operatorname{supp}(y_1)
	\cup
	\operatorname{supp}(y_2)
	\cup\{o\},
	\]
	and write
	\[
	y_j=\sum_{v\in E}a_{j,v}e_v,
	\qquad j=1,2.
	\]
	Since \(E\) and \(\{0,\ldots,n_0-1\}\) are finite, choose
	\(N\in\mathbb N\) sufficiently large so that
	\[
	\frac{1}{N}\sum_{v\in E}|a_{j,v}|<\varepsilon,
	\qquad j=1,2,
	\]
	and
	\[
	c_{|v|+k}\Phi_p(|v|+k)
	\max_{u\in\operatorname{Chi}^k(v)}
	|\lambda(v\to u)|
	\leq N
	\]
	for every \(v\in E\) and \(0\leq k<n_0\).
	
	By condition~\textup{(iii)},
	\[
	I
	:=
	\bigcap_{v\in E}
	\left\{
	n\in\mathbb N_0:
	c_{|v|+n}\Phi_p(|v|+n)
	\max_{u\in\operatorname{Chi}^n(v)}
	|\lambda(v\to u)|>N
	\right\}
	\in\mathcal F.
	\]
	The choice of \(N\) implies that every \(n\in I\) satisfies
	\(n\geq n_0\).
	
	Fix \(n\in I\). For each \(v\in E\), choose
	\(u(v,n)\in\operatorname{Chi}^n(v)\) such that
	\[
	|\lambda(v\to u(v,n))|
	=
	\max_{u\in\operatorname{Chi}^n(v)}
	|\lambda(v\to u)|.
	\]
	Define
	\[
	R_ny_j
	=
	\sum_{v\in E}
	\frac{a_{j,v}}{\lambda(v\to u(v,n))}
	e_{u(v,n)},
	\qquad j=1,2.
	\]
	Then
	\[
	B_\lambda^nR_ny_j=y_j,
	\qquad j=1,2,
	\]
	and, using the norm formula for \(e_v\),
	\[
	\begin{aligned}
		\|R_ny_j\|_{\operatorname{Ces}_p(T)}
		&\leq
		\sum_{v\in E}
		\frac{|a_{j,v}|}
		{c_{|v|+n}\Phi_p(|v|+n)
			|\lambda(v\to u(v,n))|} \\
		&<
		\frac{1}{N}\sum_{v\in E}|a_{j,v}|
		<
		\varepsilon.
	\end{aligned}
	\]
	Thus
	\[
	h_j:=x_j+R_ny_j\in U_j
	\]
	and, since \(n\geq n_0\),
	\[
	B_\lambda^nh_j=y_j\in V_j,
	\qquad j=1,2.
	\]
	Consequently,
	\[
	I
	\subseteq
	N_{B_\lambda\oplus B_\lambda}
	(U_1\times U_2,V_1\times V_2).
	\]
	Hence \(B_\lambda\oplus B_\lambda\) is
	\(\mathcal F\)-transitive.
	
	In particular, \(B_\lambda\) is weakly mixing. Moreover, for
	nonempty open sets \(U,V\subseteq\operatorname{Ces}_p(T)\),
	\[
	N_{B_\lambda}(U,V)
	=
	N_{B_\lambda\oplus B_\lambda}
	\bigl(
	U\times\operatorname{Ces}_p(T),
	V\times\operatorname{Ces}_p(T)
	\bigr).
	\]
	Hence \(B_\lambda\) is \(\mathcal F\)-transitive. By
	Lemma~\numref{lem:F-weak-mixing}, it is
	\(\widetilde{\mathcal F}\)-transitive.
\end{proof}
We shall also use the following standard notions. Let \(X\) be a
separable infinite-dimensional Banach space and let
\(S\in\mathcal L(X)\). The operator \(S\) is called \textit{weakly
mixing} if \(S\oplus S\) is hypercyclic on \(X\oplus X\), where
\[
(S\oplus S)(x,y)=(Sx,Sy),
\qquad (x,y)\in X\oplus X.
\]

A set \(A\subseteq\mathbb N_0\) is called \textit{syndetic} if there
exists \(N\in\mathbb N\) such that
\[
A\cap\{m,m+1,\ldots,m+N\}\neq\varnothing
\]
for every \(m\in\mathbb N_0\). The operator \(S\) is called
\textit{topologically ergodic} if \(N_S(U,V)\) is syndetic for every
pair of nonempty open sets \(U,V\subseteq X\).

The operator \(S\) is called \textit{topologically mixing} (or simply
\textit{mixing}) if, for every pair of nonempty open sets
\(U,V\subseteq X\), there exists \(n_0\in\mathbb N_0\) such that
\[
S^n(U)\cap V\neq\varnothing,
\qquad n\geq n_0.
\]

\begin{rem}[Classical dynamical properties and Furstenberg families]
\label{rem:classical-dynamics}

Consider the Furstenberg families
\[
\begin{aligned}
\mathcal F_\infty
&=
\{A\subseteq\mathbb N_0:
A\text{ is infinite}\},\\
\mathcal F_{\mathrm{thick}}
&=
\{A\subseteq\mathbb N_0:
A\text{ contains arbitrarily long intervals}\},\\
\mathcal F_{\mathrm{syn}}
&=
\{A\subseteq\mathbb N_0:
A\text{ is syndetic}\},\\
\mathcal F_{\mathrm{cof}}
&=
\{A\subseteq\mathbb N_0:
\mathbb N_0\setminus A\text{ is finite}\}.
\end{aligned}
\]

For \(B_\lambda\in\mathcal L(\operatorname{Ces}_p(T))\), we have
\[
\begin{aligned}
B_\lambda\text{ is hypercyclic}
&\iff
B_\lambda\text{ is }\mathcal F_\infty\text{-transitive},\\
B_\lambda\text{ is weakly mixing}
&\iff
B_\lambda\text{ is }\mathcal F_{\mathrm{thick}}\text{-transitive},\\
B_\lambda\text{ is topologically ergodic}
&\iff
B_\lambda\text{ is }\mathcal F_{\mathrm{syn}}\text{-transitive},\\
B_\lambda\text{ is topologically mixing}
&\iff
B_\lambda\text{ is }\mathcal F_{\mathrm{cof}}\text{-transitive}.
\end{aligned}
\]

These equivalences are standard; see, for example, \cite{KR2,CC}.
\end{rem}

The following example separates \(\mathcal F\)-transitivity from
\(\mathcal F\)-hypercyclicity for sets of positive lower density, even
among mixing weighted shifts.

\begin{ex}\label{ex:lower-density-transitive-not-hypercyclic}
	Consider the full binary tree \(T_2\) and set
	\(\lambda_u=1/2\) for every \(u\neq o\). Since \(c_n=2^n\), the
	boundedness criterion gives
	\[
	\sup_{n\geq0}
	\frac{c_{n+1}}{c_n}
	\max_{u\in D_{n+1}}|\lambda_u|
	=1.
	\]
	Thus \(B_\lambda\) is bounded on
	\(\operatorname{Ces}_p(T_2)\). Moreover, for \(v\in D_m\),
	\[
	c_{m+n}\Phi_p(m+n)
	\max_{u\in\operatorname{Chi}^n(v)}
	|\lambda(v\to u)|
	=
	2^m\Phi_p(m+n)
	\longrightarrow\infty
	\qquad\text{as }n\to\infty.
	\]
	Hence the set in
	Theorem~\numref{thm:F-transitivity-growth}\textup{(iii)} is cofinite.
	Therefore, \(B_\lambda\) is mixing and
	\(\mathcal F_{\underline d>0}\)-transitive.
	
	Suppose that \(B_\lambda\) is
	\(\mathcal F_{\underline d>0}\)-hypercyclic, and let \(f\) be a
	corresponding vector. Set
	\[
	W
	=
	\left\{
	g\in\operatorname{Ces}_p(T_2):
	\|g-e_o\|_{\operatorname{Ces}_p(T_2)}
	<
	\frac{1}{2\Phi_p(0)}
	\right\},
	\qquad
	A=N_{B_\lambda}(f,W).
	\]
	Then \(\underline d(A)>0\). For \(n\in A\), the pointwise estimate
	at \(o\) and the iteration formula give
	\[
	\frac12
	<
	|(B_\lambda^nf)(o)|
	=
	\biggl|
	\frac1{2^n}\sum_{u\in D_n}f(u)
	\biggr|
	\leq A_n(f).
	\]
	Choose \(0<\delta<\underline d(A)\). For all sufficiently large
	\(N\),
	\[
	\begin{aligned}
		\frac{1}{N+1}\sum_{n=0}^{N}A_n(f)
		&\geq
		\frac{1}{N+1}
		\sum_{n\in A\cap\{0,\ldots,N\}}A_n(f)\\
		&>
		\frac12
		\frac{|A\cap\{0,\ldots,N\}|}{N+1}
		>
		\frac{\delta}{2}.
	\end{aligned}
	\]
	Consequently, the series defining
	\(\|f\|_{\operatorname{Ces}_p(T_2)}^p\) diverges, a contradiction.
	Thus \(B_\lambda\) is not
	\(\mathcal F_{\underline d>0}\)-hypercyclic.
\end{ex}

\subsection{Nonzero orbit limit points}
We next consider the existence of nonzero limit points of individual
orbits, a property weaker than hypercyclicity. For weighted backward
shifts on \(\operatorname{Ces}_p(T)\), this property admits a
characterization in terms of path-weight growth.
\begin{defn}\label{def:orbit-limit-point}
	Let \(X\) be a Banach space, \(S\in\mathcal L(X)\), and \(x\in X\).
	A vector \(y\in X\) is called \textit{an orbit limit point of \(x\)} under
	\(S\) if \(S^{n_k}x\to y\) for a strictly increasing sequence of
	positive integers \((n_k)_{k\geq1}\). We write
	\[
	L_S(x)
	=
	\{y\in X:S^{n_k}x\to y
	\text{ for some }n_k\uparrow\infty\}.
	\]
	The operator \(S\) is said to have a nonzero orbit limit point if
	\[
	L_S(x)\setminus\{0\}\neq\varnothing
	\]
	for some \(x\in X\).
\end{defn}

\begin{thm}
	\label{thm:nonzero-orbit-limit}
	Let \(1<p<\infty\), let \(T=(V,E)\) be a leafless, locally finite rooted
	tree with root \(o\), and let
	\[
	B_\lambda\in
	\mathcal L\bigl(\operatorname{Ces}_p(T)\bigr)
	\]
	have nonzero weights. Then the following assertions are equivalent:
	\begin{enumerate}
		\renewcommand{\labelenumi}{\textup{(\roman{enumi})}}
		\item \(B_\lambda\) has a nonzero orbit limit point;
		\item \(e_o\in L_{B_\lambda}(f)\) for some
		\(f\in\operatorname{Ces}_p(T)\);
		\item there exist \(v\in V\) and a strictly increasing sequence of
		positive integers \((n_k)_{k\geq1}\) such that
		\[
		c_{|v|+n_k}\Phi_p(|v|+n_k)
		\max_{u\in\operatorname{Chi}^{n_k}(v)}
		|\lambda(v\to u)|
		\longrightarrow\infty;
		\]
	\item there exists a strictly increasing sequence of positive integers
	\((n_k)_{k\geq1}\) such that
	\[
	c_{n_k}\Phi_p(n_k)
	\max_{u\in\operatorname{Chi}^{n_k}(o)}
	|\lambda(o\to u)|
	\longrightarrow\infty.
	\]
	\end{enumerate}
\end{thm}

\begin{proof}
	We prove
	\[
	\textup{(ii)}\Longrightarrow\textup{(i)}
	\Longrightarrow\textup{(iii)}
	\Longrightarrow\textup{(iv)}
	\Longrightarrow\textup{(ii)}.
	\]
	The implication
	\(\textup{(ii)}\Rightarrow\textup{(i)}\) is immediate since
	\(e_o\neq0\).

	Assume \textup{(i)}. Choose \(f\in\operatorname{Ces}_p(T)\),
	\(0\neq g\in L_{B_\lambda}(f)\), and a strictly increasing sequence of
	positive integers \((n_k)_{k\geq1}\) such that
	\[
	B_\lambda^{n_k}f\to g\qquad\text{as }k\to\infty.
	\]
	Fix \(v\in V\) with \(g(v)\neq0\). By continuity of point evaluation,
	there exists \(\delta>0\) such that
	\[
	|B_\lambda^{n_k}f(v)|\geq\delta
	\]
	for all sufficiently large \(k\). Put \(r_k=|v|+n_k\). Then
	\[
	\begin{aligned}
		\delta
		\leq
		\biggl|
		\sum_{u\in\operatorname{Chi}^{n_k}(v)}
		\lambda(v\to u)f(u)
		\biggr| \leq
		c_{r_k}A_{r_k}(f)
		\max_{u\in\operatorname{Chi}^{n_k}(v)}
		|\lambda(v\to u)|.
	\end{aligned}
	\]
Consequently,
\[
c_{r_k}\Phi_p(r_k)
\max_{u\in\operatorname{Chi}^{n_k}(v)}
|\lambda(v\to u)|
\geq
\delta\frac{\Phi_p(r_k)}{A_{r_k}(f)}.
\]
Since \(r_k\to\infty\), Lemma~\numref{lem:cesaro-level-decay} gives
\[
\frac{A_{r_k}(f)}{\Phi_p(r_k)}
\longrightarrow0.
\]
Therefore,
\[
c_{r_k}\Phi_p(r_k)
\max_{u\in\operatorname{Chi}^{n_k}(v)}
|\lambda(v\to u)|
\longrightarrow\infty,
\]
which proves \textup{(iii)}.
	
	Assume \textup{(iii)}. If \(v=o\), then \textup{(iv)} is immediate.
	Otherwise, put \(m=|v|\). Since all weights are nonzero,
	\[
	|\lambda(o\to v)|>0.
	\]
	Moreover,
	\[
	\begin{aligned}
		&c_{m+n_k}\Phi_p(m+n_k)
		\max_{u\in\operatorname{Chi}^{m+n_k}(o)}
		|\lambda(o\to u)| \\
		&\quad\geq
		|\lambda(o\to v)|
		c_{|v|+n_k}\Phi_p(|v|+n_k)
		\max_{u\in\operatorname{Chi}^{n_k}(v)}
		|\lambda(v\to u)|
		\longrightarrow\infty.
	\end{aligned}
	\]
	Thus \textup{(iv)} follows with the strictly increasing sequence
\((m+n_k)_{k\geq1}\).
	
Finally, assume \textup{(iv)}, and set
\[
Q_n
=
c_n\Phi_p(n)
\max_{u\in\operatorname{Chi}^n(o)}
|\lambda(o\to u)|.
\]
By \textup{(iv)}, there exists a strictly increasing sequence
\((m_j)\) such that \(Q_{m_j}\to\infty\). Recursively choose a
subsequence \((n_k)\) of \((m_j)\) such that
\begin{equation}\label{eq:orbit-subsequence}
	Q_{n_k}>2^kC_k,
	\qquad
	C_k
	=
	\max\bigl(
	\{1\}\cup
	\{\|B_\lambda^{n_i}\|:1\leq i<k\}
	\bigr).
\end{equation}
	Choose \(u_k\in\operatorname{Chi}^{n_k}(o)\) such that
	\[
	|\lambda(o\to u_k)|
	=
	\max_{u\in\operatorname{Chi}^{n_k}(o)}
	|\lambda(o\to u)|.
	\]
	Since all weights are nonzero, \(\lambda(o\to u_k)\neq0\). Define
	\(
	h_k
	=
	\frac{1}{\lambda(o\to u_k)}e_{u_k}.
	\)
Then, by \eqnref{eq:orbit-subsequence},
	\[
	B_\lambda^{n_k}h_k=e_o,
	\qquad
	\|h_k\|_{\operatorname{Ces}_p(T)}
	=
	\frac{1}{Q_{n_k}}
	<
	\frac{1}{2^kC_k}
	\leq2^{-k}.
	\]
	Hence
	\[
	f:=\sum_{k=1}^{\infty}h_k
	\in\operatorname{Ces}_p(T).
	\]
	For fixed \(j\), we have \(B_\lambda^{n_j}h_k=0\) when \(k<j\) and
	\(B_\lambda^{n_j}h_j=e_o\). Thus, \eqnref{eq:orbit-subsequence} gives
	\[
	\begin{aligned}
		\|B_\lambda^{n_j}f-e_o\|_{\operatorname{Ces}_p(T)}
		&\leq
		\sum_{k>j}
		\|B_\lambda^{n_j}\|
		\|h_k\|_{\operatorname{Ces}_p(T)} \\
		&\leq
		\sum_{k>j}\frac{C_k}{2^kC_k}
		=
		2^{-j}
		\longrightarrow0
		\qquad (j\to\infty).
	\end{aligned}
	\]
	Thus
	\[
	B_\lambda^{n_j}f\longrightarrow e_o
	\qquad (j\to\infty),
	\]
	so \(e_o\in L_{B_\lambda}(f)\). Hence \textup{(ii)} follows.
\end{proof}
\begin{figure}[ht]
	\centering
	\includegraphics[width=.75\textwidth]
	{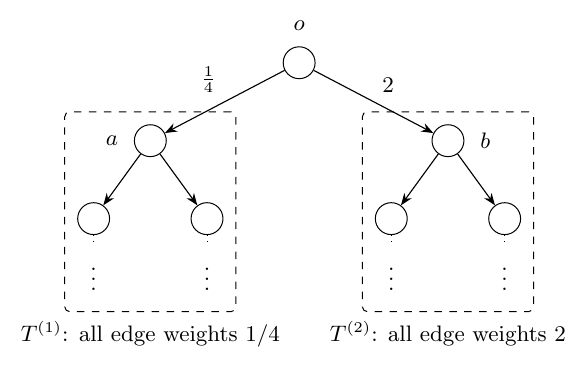}
	\caption{The weight distribution on the two principal subtrees of
		the full binary tree \(T_2\).}
	\label{fig:binary-tree-weight-split}
\end{figure}
\begin{ex}\label{ex:binary-tree-nonhypercyclic-limit-point}
	On the full binary tree, there exists a bounded weighted backward shift
	on \(\operatorname{Ces}_p(T_2)\) that is not hypercyclic but has a
	nonzero orbit limit point.
	
	Let \(T_2\) be the full binary tree rooted at \(o\), and let \(a\) and
	\(b\) be the two children of \(o\). For each \(u\neq o\), assign to the
	edge from its parent to \(u\) the weight
	\[
	\lambda_u=
	\begin{cases}
		\dfrac14,&\text{ if }\; u\in V(a),\\[2ex]
		2,&\text{ if }\; u\in V(b),
	\end{cases}
	\]
	as illustrated in Figure~\numref{fig:binary-tree-weight-split}. Since
	\(T_2\) is the full binary tree,
	\[
	c_n=2^n,\qquad n\in\mathbb N_0.
	\]
	Thus all weights in the two principal subtrees are \(1/4\) and \(2\),
	respectively. By the boundedness criterion,
	\[
	\sup_{n\geq0}
	\frac{c_{n+1}}{c_n}
	\max_{u\in D_{n+1}}|\lambda_u|
	=2\cdot2=4,
	\]
	so \(B_\lambda\in\mathcal L(\operatorname{Ces}_p(T_2))\).
	
 For \(n\geq1\) and
	\(u\in\operatorname{Chi}^n(a)\), we have
	\[
	|\lambda(a\to u)|=4^{-n}.
	\]
	Consequently,
	\[
	c_{|a|+n}\Phi_p(|a|+n)
	\max_{u\in\operatorname{Chi}^n(a)}
	|\lambda(a\to u)|
	=
	2^{1-n}\Phi_p(n+1).
	\]
	By Lemma~\numref{lem2.3},
	\[
	2^{1-n}\Phi_p(n+1)
	\leq
	(p-1)^{1/p}2^{1-n}(n+2)^{(p-1)/p}
	\longrightarrow0
	\qquad(n\to\infty).
	\]
	Hence
	\[
	\left\{
	n\in\mathbb N_0:
	c_{|a|+n}\Phi_p(|a|+n)
	\max_{u\in\operatorname{Chi}^n(a)}
	|\lambda(a\to u)|>1
	\right\}
	\]
	is finite. Taking \(F=\{a\}\) and \(N=1\) in
	Theorem~\numref{thm:F-transitivity-growth}, we see that \(B_\lambda\) is
	not \(\mathcal F_\infty\)-transitive and therefore is not hypercyclic.
	
	On the other hand, for each \(n\geq1\), choose
	\(u_n\in\operatorname{Chi}^{n-1}(b)\). Then
	\[
	|\lambda(o\to u_n)|=2^n.
	\]
	Since every edge weight is at most \(2\),
	\[
	\max_{u\in\operatorname{Chi}^n(o)}
	|\lambda(o\to u)|=2^n.
	\]
	It follows that
	\[
	c_n\Phi_p(n)
	\max_{u\in\operatorname{Chi}^n(o)}
	|\lambda(o\to u)|
	=
	4^n\Phi_p(n)
	\longrightarrow\infty
	\qquad(n\to\infty).
	\]
	By Theorem~\numref{thm:nonzero-orbit-limit}, \(B_\lambda\) has a nonzero
	orbit limit point; indeed, there exists
	\(f\in\operatorname{Ces}_p(T_2)\) such that
	\(e_o\in L_{B_\lambda}(f)\).
\end{ex}
\subsection{Chaotic weighted shifts}
We now study chaotic weighted backward shifts on
\(\operatorname{Ces}_p(T)\). Building on
\cite[Theorem~3.5]{KDP}, we characterize chaos by an explicit
summability condition for unit flows.
\begin{defn}\label{def:chaotic}
	Let \(X\) be a separable Banach space and let
	\(S\in\mathcal L(X)\). A vector \(x\in X\) is called \textit{periodic} for
	\(S\) if \(S^n x=x\) for some \(n\in\mathbb N\). The operator \(S\)
	is called \textit{chaotic} if it is hypercyclic and its periodic points are
	dense in \(X\).
\end{defn}

\begin{defn}\label{def:unit-flow}
	Let \(v\in V\). A function
	\(\theta:V(v)\to[0,\infty)\) is called a \textit{flow} on \(V(v)\)
	if
	\[
	\theta(u)
	=
	\sum_{w\in\operatorname{Chi}(u)}\theta(w),
	\qquad u\in V(v).
	\]
	A flow \(\theta\) is called a \textit{unit flow} if
	\(\theta(v)=1\).
\end{defn}

\begin{thm}\label{thm:chaos-flow-characterization}
	Let \(1<p<\infty\), let \(T=(V,E)\) be a leafless, locally finite rooted
	tree, and let
	\[
	B_\lambda\in\mathcal L(\operatorname{Ces}_p(T))
	\]
	have nonzero weights. Then the following assertions are equivalent:
	\begin{enumerate}
		\renewcommand{\labelenumi}{\textup{(\roman{enumi})}}
		\item \(B_\lambda\) is chaotic;
		
		\item the periodic points of \(B_\lambda\) are dense in
		\(\operatorname{Ces}_p(T)\);
		
		\item for every \(v\in V\), there exists
		\(f_v\in\operatorname{Ces}_p(T)\), supported on \(V(v)\), such that
		\[
		f_v(v)=1,
		\qquad
		B_\lambda f_v=f_v
		\quad\text{on }V(v);
		\]
		
		\item for every \(v\in V\), there exists a unit flow
		\(\theta:V(v)\to[0,\infty)\) such that
		\[
		\sum_{N=|v|}^{\infty}
		\biggl(
		\frac{1}{N+1}
		\sum_{s=|v|}^{N}
		\frac{1}{c_s}
		\sum_{u\in\operatorname{Chi}^{s-|v|}(v)}
		\frac{\theta(u)}{|\lambda(v\to u)|}
		\biggr)^p
		<\infty.
		\]
	\end{enumerate}
	Moreover, each of these conditions implies that \(B_\lambda\) is
	topologically mixing.
\end{thm}
\begin{proof}

We first prove
\[
\textup{(i)}\Longrightarrow\textup{(ii)}
\Longrightarrow\textup{(iii)}.
\]
We then show that \textup{(iii)} implies both \textup{(ii)} and
topological mixing, and hence \textup{(i)}. Finally, we establish the
equivalence of \textup{(iii)} and \textup{(iv)}.
	The implication \(\textup{(i)}\Rightarrow\textup{(ii)}\) follows
	directly from the definition of chaos.
	
	\medskip
	
	\noindent
	\(\textup{(ii)}\Rightarrow\textup{(iii)}\).
	Fix \(v\in V\). Since point evaluation at \(v\) is continuous, the set
	\[
	U_v
	=
	\left\{
	f\in\operatorname{Ces}_p(T):
	|f(v)-1|<\frac12
	\right\}
	\]
	is nonempty and open. By the density of the periodic points, there
	exists a periodic point \(h\in U_v\). In particular, \(h(v)\neq0\).
	Replacing \(h\) by \(h/h(v)\), we may assume that \(h(v)=1\).
	Choose \(N\in\mathbb N\) such that \(B_\lambda^Nh=h\).
	
	For a subset \(A\subseteq V\), let \(\chi_A\) denote its characteristic
	function. Let \(B_{\lambda,v}\) denote the weighted backward shift restricted to
	the subtree \(V(v)\), and set
	\[
	g
	=
	h\chi_{\bigcup_{k\geq0}\operatorname{Chi}^{kN}(v)}.
	\]
Since \(|g(u)|\leq|h(u)|\) for every \(u\in V\), we have
\[
\|g\|_{\operatorname{Ces}_p(T)}
\leq
\|h\|_{\operatorname{Ces}_p(T)}.
\]
Hence \(g\in\operatorname{Ces}_p(T)\). Moreover,
	\[
	B_{\lambda,v}^Ng=g,
	\qquad
	g(v)=1.
	\]
	Define
	\[
	f_v
	=
	g+B_{\lambda,v}g+\cdots+B_{\lambda,v}^{N-1}g.
	\]
	For \(0\leq j<N\),
	\[
	B_{\lambda,v}^jg
	=
	(B_\lambda^jg)\chi_{V(v)},
	\]
	and hence
	\[
	\|B_{\lambda,v}^jg\|_{\operatorname{Ces}_p(T)}
	\leq
	\|B_\lambda^jg\|_{\operatorname{Ces}_p(T)}
	<\infty.
	\]
	Therefore \(f_v\in\operatorname{Ces}_p(T)\). Moreover, \(f_v\) is
	supported on \(V(v)\), and
	\[
	B_{\lambda,v}f_v
	=
	\sum_{j=1}^{N}B_{\lambda,v}^jg
	=
	f_v.
	\]
	Since \(g\) vanishes on
	\(\operatorname{Chi}^j(v)\) for \(1\leq j<N\), we also have
	\[
	f_v(v)=g(v)=1.
	\]
	Thus \textup{(iii)} holds.
	
	\medskip
	
	\noindent
	\(\textup{(iii)}\Rightarrow\textup{(ii)}\).
	Fix \(v\in V\), put \(m=|v|\), and let \(f_v\) be as in
	\textup{(iii)}. Here \(\operatorname{par}^j(v)\) denotes the \(j\)-th
	ancestor of \(v\). Extend \(f_v\) to a vector
	\(\widetilde f_v\in\operatorname{Ces}_p(T)\) by setting
	\[
	\widetilde f_v(u)
	=
	\begin{cases}
		f_v(u),&u\in V(v),\\
		\lambda(u\to v),
		&u=\operatorname{par}^j(v),\ 1\leq j\leq m,\\
		0,&\text{otherwise}.
	\end{cases}
	\]
	Only finitely many new coordinates have been added, so
	\(\widetilde f_v\in\operatorname{Ces}_p(T)\). The definition of
	\(\widetilde f_v\) and the fixed-point property of \(f_v\) give
	\[
	B_\lambda\widetilde f_v=\widetilde f_v.
	\]
	
	For \(N>m\), set
	\[
	G_N
	=
	\bigcup_{k\geq0}D_{m+kN},
	\qquad
	p_N=\widetilde f_v\chi_{G_N}.
	\]
	Since \(B_\lambda\widetilde f_v=\widetilde f_v\), the iteration formula
	shows that
	\[
	B_\lambda^Np_N=p_N.
	\]
	Thus \(p_N\) is periodic. Moreover,
	\[
	p_N-e_v
	=
	\widetilde f_v
	\chi_{\bigcup_{k\geq1}D_{m+kN}}.
	\]
	Consequently,
	\[
	\|p_N-e_v\|_{\operatorname{Ces}_p(T)}
	\leq
	\left\|
	\widetilde f_v
	\chi_{\bigcup_{r\geq m+N}D_r}
	\right\|_{\operatorname{Ces}_p(T)}
	\longrightarrow0
	\qquad (N\to\infty).
	\]
	Hence every basis vector is a limit of periodic points. Since the
	periodic points form a linear subspace and the canonical basis spans a
	dense subspace, they are dense in \(\operatorname{Ces}_p(T)\).
	
	We next show that \textup{(iii)} implies mixing. Fix
	\(v\in V\), put \(m=|v|\), and let \(f_v\) be as in \textup{(iii)}.
	By Lemma~\numref{lem:cesaro-level-decay},
	\[
	\frac{A_{|v|+n}(f_v)}{\Phi_p(|v|+n)}
	\longrightarrow0
	\qquad\text{as }n\to\infty.
	\]
	
	For \(f_v\) as above and \(n\geq1\), the fixed-point identity gives
	\[
	1
	=
	|B_\lambda^nf_v(v)|
	\leq
	c_{m+n}A_{m+n}(f_v)
	\max_{u\in\operatorname{Chi}^n(v)}
	|\lambda(v\to u)|.
	\]
	Therefore,
	\[
	c_{m+n}\Phi_p(m+n)
	\max_{u\in\operatorname{Chi}^n(v)}
	|\lambda(v\to u)|
	\geq
	\frac{\Phi_p(m+n)}{A_{m+n}(f_v)}
	\longrightarrow\infty,
	\qquad n\to\infty.
	\]
Since this holds for every \(v\in V\), for every nonempty finite set
\(F\subseteq V\) and every \(M\in\mathbb N\), the set in
Theorem~\numref{thm:F-transitivity-growth}\textup{(iii)} is cofinite.
Hence \(B_\lambda\) is
\(\mathcal F_{\mathrm{cof}}\)-transitive. By
Remark~\numref{rem:classical-dynamics}, \(B_\lambda\) is
topologically mixing and therefore hypercyclic. Together with the
density of its periodic points, this shows that \(B_\lambda\) is
chaotic. Thus \textup{(i)}--\textup{(iii)} are equivalent.
	
	\medskip
	
	\noindent
	\(\textup{(iv)}\Rightarrow\textup{(iii)}\).
	Fix \(v\in V\), let \(\theta\) be the unit flow in \textup{(iv)}, and
	define
	\[
	f_v(u)
	=
	\begin{cases}
		\dfrac{\theta(u)}{\lambda(v\to u)},&u\in V(v),\\[2ex]
		0,&u\notin V(v).
	\end{cases}
	\]
	Since \(\lambda(v\to v)=1\), we have \(f_v(v)=1\). If
	\(w\in V(v)\), then
	\[
	\begin{aligned}
		B_\lambda f_v(w)
		&=
		\sum_{u\in\operatorname{Chi}(w)}
		\lambda_u\frac{\theta(u)}{\lambda(v\to u)}\\
		&=
		\frac{1}{\lambda(v\to w)}
		\sum_{u\in\operatorname{Chi}(w)}\theta(u)
		=
		\frac{\theta(w)}{\lambda(v\to w)}
		=
		f_v(w).
	\end{aligned}
	\]
	Furthermore, if \(s\geq|v|\), then
	\[
	A_s(f_v)
	=
	\frac{1}{c_s}
	\sum_{u\in\operatorname{Chi}^{s-|v|}(v)}
	\frac{\theta(u)}{|\lambda(v\to u)|}.
	\]
	Hence condition \textup{(iv)} is precisely
	\[
	\|f_v\|_{\operatorname{Ces}_p(T)}^p<\infty.
	\]
	Thus \(f_v\in\operatorname{Ces}_p(T)\), and \textup{(iii)} follows.
	
	\medskip
	
	\noindent
	\(\textup{(iii)}\Rightarrow\textup{(iv)}\).
	Fix \(v\in V\), and let \(f_v\) satisfy \textup{(iii)}. For
	\(u\in V(v)\), put
	\[
	q(u)=\lambda(v\to u)f_v(u).
	\]
	Then
	\[
	q(v)=\lambda(v\to v)f_v(v)=1.
	\]
	Moreover, for every \(w\in V(v)\),
\begin{equation}\label{eq:q-flow-identity}
	\begin{aligned}
		q(w)
		&=\lambda(v\to w)f_v(w)
		=\lambda(v\to w)(B_\lambda f_v)(w)\\
		&=\sum_{u\in\operatorname{Chi}(w)}
		\lambda(v\to w)\lambda_u f_v(u)
		=\sum_{u\in\operatorname{Chi}(w)}q(u).
	\end{aligned}
\end{equation}
	
	We construct \(\theta\) level by level. Set \(\theta(v)=1\). Since
	\(q(v)=1\), we have \(0\leq\theta(v)\leq|q(v)|\). Suppose that
	\(\theta\) has been defined on
	\[
	\bigcup_{j=0}^{n}\operatorname{Chi}^j(v)
	\]
	and satisfies \(0\leq\theta(w)\leq|q(w)|\) there. For
	\(w\in\operatorname{Chi}^n(v)\), put
	\[
	S(w)=\sum_{u\in\operatorname{Chi}(w)}|q(u)|.
	\]
By~\eqnref{eq:q-flow-identity}, the induction hypothesis, and the triangle inequality,
\[
\theta(w)\leq |q(w)|\leq S(w).
\]

	For \(u\in\operatorname{Chi}(w)\), define
	\[
	\theta(u)
	=
	\begin{cases}
		\dfrac{\theta(w)|q(u)|}{S(w)},&S(w)>0,\\[2ex]
		0,&S(w)=0.
	\end{cases}
	\]
	If \(S(w)=0\), then \(\theta(w)=0\). Hence, in both cases,
	\[
	\sum_{u\in\operatorname{Chi}(w)}\theta(u)=\theta(w),
	\qquad
	0\leq\theta(u)\leq|q(u)|.
	\]
	Since every vertex in \(V(v)\setminus\{v\}\) has a unique parent, this
	defines \(\theta\) unambiguously on the next generation. Inductively,
	\(\theta\) is a unit flow on \(V(v)\) satisfying
\begin{equation}\label{eq:theta-dominance}
	0\leq\theta(u)\leq|q(u)|,
	\qquad u\in V(v).
\end{equation}
By \eqnref{eq:theta-dominance},
\[
\frac{\theta(u)}{|\lambda(v\to u)|}
\leq
\frac{|q(u)|}{|\lambda(v\to u)|}
=
|f_v(u)|.
\]
	It follows that
	\[
	\begin{aligned}
		&\sum_{N=|v|}^{\infty}
		\biggl(
		\frac{1}{N+1}
		\sum_{s=|v|}^{N}
		\frac{1}{c_s}
		\sum_{u\in\operatorname{Chi}^{s-|v|}(v)}
		\frac{\theta(u)}{|\lambda(v\to u)|}
		\biggr)^p\\
		&\qquad\leq
		\sum_{N=|v|}^{\infty}
		\biggl(
		\frac{1}{N+1}
		\sum_{s=|v|}^{N}A_s(f_v)
		\biggr)^p
		=
		\|f_v\|_{\operatorname{Ces}_p(T)}^p
		<\infty.
	\end{aligned}
	\]
	Thus \textup{(iv)} holds, completing the proof.
\end{proof}
The preceding characterizations yield a sharp threshold for constant
weights on homogeneous trees.
\begin{ex}\label{ex:homogeneous-constant-weight}
	Let \(T_q\) be the full \(q\)-ary tree, \(q\geq2\), and set
	\(\lambda_u=\alpha\neq0\) for \(u\neq o\). Since \(c_n=q^n\),
	Theorem~\numref{thm:boundedness} gives
	\(B_\lambda\in\mathcal L(\operatorname{Ces}_p(T_q))\).
	For \(v\in D_m\),
	\[
	c_{m+n}\Phi_p(m+n)
	\max_{u\in\operatorname{Chi}^n(v)}
	|\lambda(v\to u)|
	=
	q^m(q|\alpha|)^n\Phi_p(m+n).
	\]
	Hence \(B_\lambda\) is mixing if and only if \(q|\alpha|\geq1\).
	Moreover, every unit flow \(\theta\) on \(V(v)\) satisfies
	\[
	\frac{1}{c_{m+n}}
	\sum_{u\in\operatorname{Chi}^n(v)}
	\frac{\theta(u)}{|\lambda(v\to u)|}
	=
	q^{-m}(q|\alpha|)^{-n}.
	\]
	Thus Theorem~\numref{thm:chaos-flow-characterization} shows that
	\(B_\lambda\) is chaotic if and only if \(q|\alpha|>1\). In
	particular, \(|\alpha|=1/q\) gives a mixing but nonchaotic shift.
\end{ex}
The critical case \(q|\alpha|=1\) therefore provides a simple
separation between topological mixing and chaos.\\

\noindent\textbf{Conflict of Interest.}
The authors declare that they have no competing interests.

\noindent\textbf{Funding.}
This work was supported by the National Natural Science Foundation of
China Grant No. 12571088.

\noindent\textbf{Data Availability Statement.}
No datasets were generated or analyzed in this theoretical study.

% ------------------------------------------------------------------------
\end{document}